\documentclass{amsart}

\usepackage{hyperref}

\usepackage[utf8]{inputenc}
\usepackage[T1]{fontenc}

\usepackage{mathrsfs}

\usepackage{amsmath, amsthm, amsfonts, amscd, amssymb}

\usepackage{commath}

\usepackage[all,cmtip]{xy}

\usepackage{tikz-cd}
 \usetikzlibrary{positioning}

\usepackage{mathtools}

\usepackage[utf8]{inputenc}

\usepackage{pdfpages}

\newtheorem{lem}{Lemma}[section]
\newtheorem{teo}[lem]{Theorem}
\newtheorem{pro}[lem]{Proposition}
\newtheorem{cor}[lem]{Corollary}
\newtheorem{claim}[lem]{Claim}

\newtheorem*{rem*}{Remark}
\newtheorem*{teo*}{Theorem}

\newcounter{claimcounter}
\numberwithin{claimcounter}{lem}

\newcommand{\tens}[1]{%
  \mathbin{\mathop{\otimes}\limits_{#1}}%
}

\DeclareMathOperator{\D}{\mathcal D}

\DeclareMathOperator{\im}{Im}

\DeclareMathOperator{\ab}{ab}

\DeclareMathOperator{\Tor}{Tor}

\newcommand{\Z}{\mathbb{Z}}
\newcommand{\F}{\mathbb{F}}

\newcommand{\Q}{\mathbb{Q}}

\renewcommand{\hat}{\widehat}
\newcommand{\ti}{\widetilde}
\renewcommand{\tilde}{\widetilde}

\renewcommand{\bar}{\overline}
\renewcommand{\epsilon}{\varepsilon}

\newcommand{\lrar}{\longrightarrow}

\newcommand{\rar}{\rightarrow}

\newcommand{\lan}{\langle}
\newcommand{\ran}{\rangle}
\newcommand{\inc}{\xhookrightarrow{}}

\newcommand{\al}{\alpha}

\newcommand{\ga}{\gamma}
\newcommand{\Ga}{\Gamma}

\newcommand{\p}{\F_p}

\newcommand{\abr}{r_{\ab}}

\DeclareMathOperator{\ore}{ore}
\newcommand{\n}{\unlhd}

\newcommand{\hp}{\hat{p}}

\newcommand{\sagecode}[1]{}

\newcommand{\FF}{{\bf F}}
\newcommand{\GG}{{\bf G}}

\newcommand{\NN}{{\bf N}}

\renewcommand{\subset}{\subseteq}
\newcommand{\sub}{\subseteq}

 \date{\today}
 
 
\title[Parafree groups]{Parafree  graphs of groups with cyclic edge  groups}
\author{Andrei Jaikin-Zapirain}

\address{Departamento de Matem\'aticas, Universidad Aut\'onoma de Madrid \and  Instituto de Ciencias Matem\'aticas, CSIC-UAM-UC3M-UCM}
\email{andrei.jaikin@uam.es}
\author{Ismael Morales}
 
\address{Mathematical Institute, University of Oxford, Radcliffe Observatory, Andrew Wiles Building, Woodstock Rd, Oxford OX2 6GG}
\email{morales@maths.ox.ac.uk}

\begin{document}

\begin{abstract} We establish a combination theorem for parafree groups. These groups were introduced by Baumslag in the sixties. One of the current motivations for a better understanding of their structure is that they show up naturally in connection with Remeslennikov's conjecture on the profinite rigidity of free groups. In this article, we determine when the fundamental group  of a finite graph of groups with cyclic edge groups  is parafree.   
\end{abstract}
 \maketitle
 
 \section{Introduction}
 
A group   is said to be  \textit{parafree} if it  is residually nilpotent  and  its quotients by the terms of its lower central series are the same as those of a free group. These groups were introduced by Baumslag \cite{Ba67}, who also constructed first examples of non-free parafree groups. From now on, we will assume  that finite generation  part of the definition of parafree groups. 

Another motivation for studying parafree groups relates to the following central open problem in the theme of profinite rigidity, attributed to Remeslennikov \cite[Question 12]{Nos79}: given a finitely generated residually finite group $G$ with the same profinite completion as a free group $F$, does it follow that $G\cong F$? The first author \cite{Ja20} showed that such $G$ must be a parafree group.

Many known examples of parafree groups are isomorphic to amalgamated free products or HNN extensions of free groups  over cyclic edge groups \cite{Ba05}. 
The purpose of this paper is to characterize when these constructions yield to parafree groups.

Our first result describes the exact circumstances under which an amalgamated free product   with cyclic  amalgam is parafree. 

\begin{teo} \label{amalgamparafreeintr} Let $U$ and $V$ be   finitely generated groups,  $1\neq u\in U$ and $1\neq v\in V$. Consider the amalgamated free product  $W=U\underset{u=v}{*}V$.
Then $W$ is parafree if and only if the following three conditions hold.
\begin{enumerate}
\item The groups $U$ and $V$ are parafree.
    \item  The element $uv^{-1}$ is not a proper power in  the abelianization   of  $U*V$.
    \item At least one of $u$ or $v$ is not a proper power in $U$ or $V$, respectively.
\end{enumerate} 
\end{teo}

A particular case  of Theorem \ref{amalgamparafreeintr}, where $U$ and $V$ are free, follows from the work of Baumslag \cite{Ba68} and Azarov \cite{Az98}. The case of HNN extensions with cyclic base groups is more intricate, and we obtain the following (less explicit) result.

\begin{teo} \label{hnnparafreeintr} Let $U$ be a  finitely generated group, $u\in U\setminus \{1\}$ and $\alpha:\langle u\rangle \to U$ a monomorphism. Put $v=\alpha(u)$.  
Consider the   HNN extension  $W=U *_{\alpha}$.
Then $W$ is  parafree if and only if the following four  conditions hold.
\begin{enumerate}
\item The group $U$ is parafree.
    \item The element $uv^{-1}$ is not a proper power in the abelianization of $U$.
    \item At least one of $u$ or $v$ is not a proper power in $U$.
    \item The image of the element $u$ is non-trivial in some finite nilpotent quotient of $W$.
\end{enumerate}
 \end{teo}
In the case where  $U$ is a parafree group with abelianisation isomorphic to $\Z^2$, the last condition of the previous theorem can be replaced by a simpler one.

\begin{cor} \label{hnncor} Let $U$ be a finitely generated group with $\abr (U)=2$, $u \in U\setminus \{1\}$ and $\alpha:\langle u\rangle \to U$ a monomorphism.  Put  $v=\alpha(u)$. 
Consider the   HNN extension  $W=U*_{\alpha}$.
Then $W$ is  parafree if and only if the following four  conditions hold.
\begin{enumerate}
    \item The group $U$ is parafree.
    \item The element $uv^{-1}$ is not a proper power in the abelianization of $U$.
    \item At least one of $u$ or $v$ is not a proper power in $U$.
    \item The images of $u$ and $v$  in the abelianization of $U$ generate a subgroup isomorphic to $\Z^2$.
\end{enumerate}
 \end{cor}

 It is tempting to try to formulate a general criterion for the fundamental group of a graph of groups with   cyclic edge groups to be parafree.  Combining Theorems \ref{amalgamparafreeintr} and  \ref{hnnparafreeintr}, we derive  the following result. 
 \begin{cor} \label{fundamentalparafreeintr} Let $(\mathcal G, \Gamma)$ be a  graph of   groups over a finite 
graph $\Gamma$ and  $W=\pi_1(\mathcal G, \Gamma)$ be its  fundamental group. Assume that all  vertex subgroups  $\mathcal G(v)$ ($v\in V(\Gamma)$)   are finitely generated and all   edge subgroups $\mathcal G(e)$ ($e\in E(\Gamma)$)  are cyclic. Then $W$ is parafree if and only if the following four conditions hold.
 
 \begin{enumerate}
 \item All the vertex  subgroups $\mathcal G(v)$ ($v\in V(\Gamma)$)    are parafree.
 
 \item The abelianization of $W$ is torsion-free   of rank 
 $$\abr(W)=\sum_{v\in V(\Gamma)}\abr(\mathcal G(v))-\sum_{e\in E(\Gamma)} \abr(\mathcal G(e)) +1-\chi(\Gamma),$$   where $\chi(\Gamma)=|V(\Gamma)|-|E(\Gamma)|$.

 \item All the centralizers of non-trivial elements in $W$ are cyclic.

 \item For  each non-trivial edge subgroup of  $\mathcal G(e)$  ($e\in E(\Gamma)$) there is a finite nilpotent quotient of $W$ where the image of this edge subgroup is non-trivial.
 \end{enumerate}

 \end{cor}

It is relatively easy to see that the conditions presented in Theorem \ref{amalgamparafreeintr} and Theorem \ref{hnnparafreeintr} are necessary for $W$ to be  parafree and they  also imply that the quotients of $W$ by the terms of its lower central series are the same as those of a free group. The difficult part in the proofs of both theorems is to show that $W$ is residually nilpotent.
We establish this  as an application of recent methods developed in \cite{Ja20} which were used for constructing  abstract subgroups of  free pro-$p$ groups.

The paper is organized as follows. In Section \ref{prelim} we describe the preliminary results which we use in the paper. Section \ref{torfree} is devoted to construction of new examples of $\D_{\p G}$-torsion-free $\p G$-modules. In Section \ref{finalsect} we prove Theorem \ref{amalgamparafreeintr} and Theorem \ref{hnnparafreeintr} and their corollaries.

\section*{Acknowledgments}
This paper is partially supported by the  Spanish Ministry of Science and Innovation through the grant   MTM2017-82690-P and the ``Severo Ochoa Programme for Centres of Excellence in R\&D"  (CEX2019-000904-S4).  The work of the second author is supported  by the collaboration grant JAE Intro SOMdM 2020, associated to ICMAT. The authors would like to thank an anonymous referee for their comments and Martin Bridson for pointing out the paper \cite{Az98}.

\section{Preliminaries}\label{prelim}

\subsection{General notation for  rings, groups and pro-$p$ groups}
All of our rings $R$ will be associative and unitary. All ring homomorphisms will map $1\mapsto 1$. All $R$-modules are left $R$-modules, unless we say that we consider right $R$-modules.

Given a group $G$, we will denote by $G'=[G, G]$ its commutator subgroup
 and by $G_{\ab}=G/G'$ its abelianization. The terms of the lower central series of $G$ are defined as $\gamma_1(G)=G$ and $\gamma_{n+1}(G)=[\gamma_n(G), G]$ if $n\geq 1$ and the terms of the $p$-lower central series of $G$ are defined as $G_{1,p}=G$ and $G_{n+1,p}=[ G_{n,p}, G]G_{n,p}^p$ if $n\geq 1$. We recall that when $\GG$ is a finitely generated pro-$p$ group, the groups $\ga_k \GG$ are closed normal subgroups of $\GG$.  
 
 We denote by $d(G)$ the minimal number of generators of a group $G$ and we put  $\abr(G)=d(G_{\ab})$. Similarly, given a pro-$p$ group $\GG$, we denote by $d(\GG)$ the  minimal number of topological generators of $\GG$, which also equals the dimension of its $p$-abelianisation $\GG/[\GG, \GG]\GG^p$ as an $\F_p$-vector space.
 
As usual, we  denote the free product in the category of pro-$p$ groups by $\coprod$.

\subsection{Parafree groups}

If $G$ is a group we denote by $G_{\hat{p}}$ the pro-$p$ completion of $G$. In this paper we will use the following characterization of   parafree groups.

\begin{pro} \label{carparafree} Let $G$ be finitely generated residually nilpotent group. Then $G$ is parafree if and only if $G_{\hat{p}}$ is a free pro-$p$ group for every prime $p$. Moreover, a parafree group  is residually-$p$ for every prime $p$.
\end{pro}

\begin{proof}A group is said to be weakly $p$-parafree if its quotients by the terms of its $p$-lower central series are the same as those of a free group. In these terms, a finitely generated $G$ will be parafree if and only if it is residually nilpotent and it is weakly $p$-parafree for every prime $p$. We know from \cite[Corollary 2.9]{Lac10} that a finitely generated $G$ is  weakly $p$-parafree if and only if it has free pro-$p$ completion. This proves the first claim. The second claim of the proposition follows from the fact that finitely generated torsion-free nilpotent groups are residually-$p$ for every prime $p$ \cite[Theorem 2.1]{Gru57}. 
\end{proof}
The following result provides a useful necessary condition for a group to be parafree.
\begin{pro}  \label{centralizer} 
The centralizers of non-trivial elements of a parafree group are cyclic.
 \end{pro}
 \begin{proof}  Let $A$ be  the centralizer of a non-trivial element $u$ of a parafree group $G$. By
\cite[Theorem 4.2]{Bau69}, we know that
\begin{enumerate}
\item[(a)] a  two-generated subgroup of $G$ is free and 
\item[(b)] an abelian subgroup of $G$ is cyclic. 
\end{enumerate} Thus, by (a), $A$ is locally cyclic,   and, by (b), $A$ is cyclic. \end{proof}

\subsection{Group algebra and the augmentation ideal}

Let $G$ be a group and $k$ a commutative ring.  The group ring of $G$ over $k$ is denoted by $kG$. We denote by $I_G$  the  augmentation ideal of  $\Z G$ and by $kI_G=k\tens{\Z} I_G$ the augmentation ideal of $kG$.  


Given two groups $H\leq G$, we denote by $kI_H^{G}$ the left ideal of $kG$ generated by $kI_{H}$. Since $kG$ is a natural free right $kH$-module, it follows that the canonical map of $kG$-modules 
\[kG\tens{kH} k I_{H}\lrar kI_{H}^{G}\]
is an isomorphism.
\subsection{The fundamental groups of graphs of groups}

We refer the reader to \cite{DD89} for standard definitions and notions related to graph of groups and their fundamental groups. 

By a \textit{graph of groups}  $(\mathcal G, \Gamma)$ we mean a connected graph $\Gamma =(V(\Gamma), E(\Gamma), \iota, \tau)$ together with a function $\mathcal G$ which assigns to each $v\in V(\Gamma)$ a group $\mathcal G(v)$, and to each $e\in E(\Gamma)$ a distinguished subgroup $\mathcal G(e)$ of $\mathcal G(\iota(e))$ and injective group homomorphism $t_e:\mathcal G(e)\to \mathcal G(\tau(e))$. The monomorphisms $t_e$ are called the \textit{edge functions}.

In this paper we work mostly with two particular cases:  amalgamated free products and HNN extensions. The structure of the augmentation ideal of the group algebra of the  fundamental group of a graph of groups is described in \cite[Lemma 6]{Ch76}. We will describe this result in  the cases that interest us.
\begin{pro} \label{straugm} Let $k$ be a commutative ring. 
\begin{enumerate} 
\item Let $U$ and $V$ be two groups, $A$ a subgroup of $U$ and $\alpha: A\to V$ a monomorphism. Put $W=U*_A V$. Then there exists an exact of sequence of $kW$-modules.
$$0\to kI_A^W\xrightarrow{\gamma} kI_U^W\oplus kI_V^W\xrightarrow{p} kI_W\to 0,$$
where $\gamma(a)=(a,-a)$ if $a\in kI_A$ and $p(b,c)=b+c$ if $b\in kI_U^W$ and $c\in kI_V^W$.

\item  Let $U$  be a group, $A$ a subgroup of $U$ and $\alpha: A\to U$ a monomorphism. Put $W=U*_\alpha=\langle U, t: tg=\alpha(g) t \textrm{\ if  }g\in A\rangle $.
Then there exists an exact of sequence of $kW$-modules.
$$0\to kI_{\alpha(A)}^W\xrightarrow{\gamma} kI_U^W\oplus kW \xrightarrow{p}k I_W\to 0,$$
where $\gamma(a)=( a(1-t),a)$ if $a\in k I_{\alpha(A)}^W$ and  $p(b,c)=b+c(t-1)$ if $b\in kI_U^W$ and $c\in kW$.
 
\end{enumerate}
 
\end{pro}

The following proposition will be used several times in the paper.

\begin{pro} \cite[Corollary 1.14 and Proposition 1.20]{Ba93}  \label{subgraph} Let $(\mathcal H, \Delta)$ be a subgraph of groups of a graph of groups $(\mathcal G, \Gamma)$, i.e. 
\begin{enumerate}
\item $\Delta$ is  a connected subgraph of a finite connected graph  $\Gamma$;
\item $\mathcal H(v)\le \mathcal G(v)$ if $v\in V(\Delta)$, $\mathcal H(e)\le \mathcal G(e)$ if $e\in E(\Delta)$ and the edge functions of $(\mathcal H, \Delta)$ are the restrictions of the edge functions of $(\mathcal G, \Gamma)$.
\end{enumerate}
Assume that $$\mathcal H(e)=\mathcal G(e)\cap \mathcal H(\iota(e)) \textrm{\  and\ } t_e(\mathcal H(e))=t_e(\mathcal G(e))\cap \mathcal H(\tau(e))  \textrm{\ if\ }e\in E(\Delta).$$ 
 Let $\Delta_0$ be a maximal subtree in $\Delta$ and $\Delta_0\subset \Gamma_0$ a maximal subtree in $\Gamma$. Then the canonical map between the fundamental groups  of graphs of groups 
$$\pi_1(\mathcal H,\Delta, \Delta_0)\to \pi_1(\mathcal G,\Gamma,\Gamma_0)$$ is injective.

\end{pro}

\subsection{The induced map between the augmentation ideals} Given a group homomorphism $\tilde G\to G$, we obtain an induced map between the augmentation ideals
 $I_{\tilde G}\to I_G$.
Here we  will explain how one can derive information from   this induced map back to the initial group homomorphism. 

\begin{pro} \label{kernelmodp} Let $\phi : \ti{G}\rar G$ be a surjective group homomorphism of kernel $K$. Suppose that the natural  homomorphism of $k G$-modules 
\[\alpha: k G  \tens{k \ti{G}} kI_{\ti{G}}    \lrar k I_G,\]
defined by the $k$-linear extension of $a \otimes b\mapsto a\phi(b)$, is an isomorphism. Then 
\[k\otimes_{\Z} K_{\ab}=0.\]
\end{pro} 
\begin{proof} Observe that by Shapiro's lemma 
$$k\otimes_{\Z} K_{\ab}\cong H_1(K,k) \cong H_1(\ti{G}, kG)=\Tor_1^{k\ti{G}}( kG,k).$$ Applying the right-exact functor $kG \otimes_{k\ti{G}} $ to the exact sequence  of  left $k\ti{G}$-modules
$$0\to kI_{\ti{G}}\to k\ti{G}\to k\to 0,$$  we get  the exact sequence

$$0\to \Tor_1^{k\ti{G}}(kG,k)\to  k{G} \otimes_{k\ti{G}} kI_{\ti{G}} \xrightarrow{\alpha} kG\to k\to 0.$$

Since $\alpha$ is injective,  $ \Tor_1^{k\ti{G}}(kG,k)=\{0\}$ and we are done.
\end{proof}


 
\subsection{$\D$-torsion-free modules}
Let $R\inc \D$ be an embedding of the ring $R$ into a division ring $\D$. Let $M$ be an $R$-module. We say that $M$ is \textit{$\D$-torsion-free} if the canonical map $M\lrar \D\tens{R}M$ is injective.

The following provides a more flexible criterion for verifying whether a module is torsion-free. 
\begin{lem} \cite[Lemma 4.1]{Ja20} \label{tfree}
Let    $M$ be an $R$-module. Then $M$ is $\D$-torsion-free if and only if there exists a $\D$-module $N$ and an injective homomorphism of $R$-modules $M\inc N$. 
\end{lem}

Given a homomorphism $R\to \D$, where $\D$ is a division ring, we define \textit{$\D$-dimension of $M$} to be  the dimension of the $\D$-vector space $\D\otimes_R M$. We denote it by $\dim_{\D} M$. 

\begin{lem}\label{sestorfree} \cite[Lemma 4.2]{Ja20}
 Let $1\rar M_1\rar M_2\rar M_3\rar 0$ be an exact sequence of $R$-modules. Assume that 
\begin{enumerate}
    \item $M_1$ and $M_3$ are $\D$-torsion-free;
    \item $\dim_{\D}M_1$ and $\dim_{\D}M_3$ are finite;  and
    \item $\dim_{\D} M_1+\dim_{\D} M_3=\dim_{\D} M_2$. 
\end{enumerate}
Then $M_2$ is also $\D$-torsion-free.
\end{lem}

The following lemma will be used in the computation of $\dim _{\D} M$.

\begin{lem} \label{quotientdim} \cite[Lemma 4.3]{Ja20}
Let $M$ be a $\D$-torsion-free $R$-module of finite $\D$-dimension. Let $L$ be a non-trivial $R$-submodule of $M$. Then $\dim_{\D}(M/L)<\dim_{\D} (M)$. Moreover, if $\dim_{\D} L=1$, then $\dim_{\D} (M/L)=\dim_{\D}M -1$. 
\end{lem}

\subsection{The universal division rings of fractions of the  group algebra of a subgroup of a free pro-$p$ group} \label{universalsf}
Let $\FF$ be a finitely generated free pro-$p$ group. If  $G$ is an abstract subgroup of $\FF$, then 
it turns out that   $\F_p G$ has  a universal division ring of fractions,  denoted    by   $\D_{\p G}$.
 We will not give a formal definition of universal field of fractions which can be found in \cite{Ja21}, but we will describe the main properties of the embeddings $\p G\subset \D_{\p G}$, which we will use in this paper.
 
 If $H$ is a subgroup of $G$, then the division  closure of $\F_p H$ in $\D_{\p G}$ is the universal division ring of fractions of $\p H$, and, therefore, we will denote it by $\D_{\p H}$.
 
  If $N$ is a normal subgroup of $G$ such that $G/N\cong \Z$, then $\F_p G$ is a free $\F_p N$-module with the basis $\{t^i:i\in \Z\}$, where $t\in G$ and $Nt$ generates $G/N$. Thus  $\p G$ is isomorphic to a skew-polynomial ring  $\F_p N[t^{\pm 1}, \sigma]$ with the indeterminate $t$ and  coefficients in $\p N$. Here $\sigma$ is an automorphism of the ring  ${\p N}$, induced from the automorphism of conjugation-by-$t$.

 Let $R$ be a subring of $\D_{\p G}$ generated by $\D_{\p N}$ and $t$. It turns out that $\{t^i:i\in \Z\}$ is also a $\D_{\p N}$-basis of $R$ (this is so called  the \textit{Hughes-free property} of the embedding $\p G\subset \D_{\p G}$, see \cite{Hu70, DHS08}). thus  $R$ is isomorphic to  $\D_{\p N}[t^{\pm 1}, \sigma]$. 
 
 The following result gives a lower bound for the $\D_{\p G}$-dimension of $\p I_G$.
 \begin{pro}\cite[Corollary 3.7 and the discussion afterwards]{Ja20} \label{Bettibound}
 Let $G$ be a finitely generated dense subgroup of a free pro-$p$ group $\FF$. Then $\dim_{\D_{\p G}} \p I_G\ge d(\FF).$
 Moreover, if $G$ is parafree and $\FF=G_{\hat p}$, then  $\dim_{\D_{\p G}} \p I_G= d(\FF).$
 \end{pro}

The  following result provides an important example of $\D_{\F_p G}$-torsion-free-module.

\begin{pro} \label{Atorfree} \cite[Proposition 4.8]{Ja20}  Let $\FF$ be a free pro-$p$ group,  $H\leq G\leq \FF$ two  subgroups of  $\FF$, and $A$ a maximal abelian subgroup of $H$. Then the $\F_pG$-module
$\F_p I_H^G\, \big/\, \F_p I_A^G$
is $\D_{\F_p G}$-torsion-free. 
\end{pro}

\subsection{Embeddings of abstract groups into free pro-$p$ groups}
In this subsection we detail a method  introduced in \cite{Ja20} for ensuring that a map from an abstract group $\ti{G}$ to a free pro-$p$ group is injective. We are particularly interested in the problem of producing families of parafree groups. Let  $\ti{G}$ be a candidate to being parafree, meaning that $\ti{G}_{\hat{p}}$ is free for every prime $p$. We want to study whether the canonical map $\ti{G}\lrar \ti{G}_{\hat{p}}$ is an embedding for some suitable prime $p$. This would establish the residual nilpotence of $\ti{G}$ and we could conclude that $\ti{G}$ is parafree.

\begin{pro} \label{pro-$p$ embedding lemma} Let $\ti{G}$ be a finitely generated group, $\FF$  a finitely generated free pro-$p$ group and $\phi : \ti{G}\rar  \FF$   a group homomorphism. Suppose that we have the following conditions. 
\begin{enumerate}
    \item The image $G=\phi(\ti{G})$ is dense in $\FF$.
    \item The $\F_p G$-module $ \F_p G \tens{\F_p \ti{G}}\p I_{\ti{G}} $ is  $\D_{\F_p G}$-torsion-free  and $$\dim_{\D_{\F_p G}} \F_p G \tens{\F_p \ti{G}} \p I_{\ti{G}}=d  (\FF).$$
    \item The kernel of $\phi$ is free. 
\end{enumerate}
 Then the map $\phi$ is an embedding. 
\end{pro}
\begin{proof}
We will prove that the surjective map $\phi: \ti{G}\rar G$ verifies the assumption of Proposition  \ref{kernelmodp} to deduce that $\ker \phi$ has trivial $p$-abelianization. Since $\ker \phi$ is free, the latter would imply that $\ker \phi=1$ and the conclusion would follow. 

It is clear that the natural  homomorphism of $\F_p G$-modules 
\[\F_p G \tens{\F_p \ti{G}} I_{\ti{G}}\rar \F_p I_{G},\]
defined by $a \otimes b\mapsto a\phi(b)$, is  surjective. If it was not injective, naming its kernel by $L$ and naming $M=\F_p G \tens{\F_p \ti{G}} I_{\ti{G}}$, we would deduce, after applying Lemma \ref{quotientdim}, \[d(\FF)= \dim_{\D_{\F_p G}} M>\dim_{\D_{\F_p G}}(M/L)=\dim_{\D_{\F_p G}}(\F_p I_{G}),\]
which would contradict Proposition \ref{Bettibound}. 
\end{proof}

In the setting of our problem, that is, taking some $\ti{G}$ with free pro-$p$ completion and studying whether $\ti{G}\lrar \ti{G}_{\hat{p}}$
is injective, we make a few comments about the three conditions of Proposition \ref{pro-$p$ embedding lemma}. We take $G$ to be the image of $\ti{G}$ inside  $\ti{G}_{\hat{p}} $.

\begin{enumerate}
    \item The first condition will be naturally ensured. 
    \item The second condition   is the hardest part and requires  the most technical arguments presented in  Section \ref{completedsection}.
    \item The third condition is natural from the point of view of the Bass-Serre theory. If we take   $\ti{G}$ to be the fundamental group of a graph of groups, and we ensure that  $\ker \phi$ intersects trivially every vertex subgroup, then the $\ker \phi$ will necessarily be free. 
\end{enumerate}

 \section{Examples of $\D_{\F_p G}$-torsion-free modules}\label{completedsection} \label{torfree}

In this section we construct two families of examples of  $\D_{\F_p G}$-torsion-free modules. The first result is a slight generalization of \cite[Proposition 4.10]{Ja20}.

\begin{pro} \label{amaltor} Let $H_1$ and $H_2$ be two finitely generated subgroups of a finitely generated free pro-$p$ group $\FF$. Consider $A=H_1\cap H_2$ and suppose that $A$ is a maximal abelian subgroup of $H_1$. Let $G=\lan H_1, H_2\ran$ and let 
\[J=\{(x, -x): x\in \p I_A^{G}\}\leq \p I_{H_1}^G\oplus \p I_{H_2}^G.\]
Then the $\p G$-module 
\[M=\frac{\p I_{H_1}^G\oplus \p I_{H_2}^G}{J}\]
is $\D_{\p G}$-torsion-free and \[\dim_{\D_{\p G}} M=\dim_{\D_{\p H_1}}\p I_{H_1}+\dim_{\D_{\p H_2}} \p I_{H_2}-1.\]
\end{pro}

Before giving the proof, we shall make an observation. If $A$ is an  abelian subgroup of $\FF$, as occurs in the previous proposition,  the universal division $\p A$-ring of fractions $\D_{\p A}$ is the field of fractions $Q_{ore}(\p A)$ of the commutative domain $\p A$. Given a nontrivial ideal $I$ of  $ \p A$, we see that the $Q_{\ore}(\p A)$-vector space $Q_{\ore}(\p A)\tens{\p A} I$ is one-dimensional.   This implies that 
\begin{equation}\label{dimA}
\dim_{\D_{\p A}} I=\dim_{Q_{\ore}(\p A)} Q_{\ore}(\p A)\tens{\p A} I=1.
\end{equation}

\begin{proof}[Proof of Proposition \ref{amaltor}] We consider the $\p G$-submodule of $M$ defined by
 $$L=(\p I_{A}^G\oplus \p I_{H_2}^G)/J.$$
We want to apply Lemma  \ref{sestorfree} to the short exact sequence $$0\rar L\rar M\rar M/L\rar 0$$ of $\p G$-modules. 
Since $L\cong \p I_{H_2}^G\leq \p G$, then it is $\D_{\p G}$-torsion-free and 
\[\dim_{\D_{\p G}} L=\dim_{\D_{\p H_2}} \p I_{H_2}.\] 
By the observation (\ref{dimA}), $$1=\dim_{\D_{\p A}} \p I_A=\dim_{\D_{\p G}} \p I_A^G=\dim_{\D_ {\p G}} J.$$ 
So it is clear, by Lemma  \ref{quotientdim}, that 
\begin{multline*}
\dim_{\D_{\p G}} M= \dim_{\D_{\p G}} \p I_{H_1}^G+\dim_{\D_{\p G} }\p I_{H_2}^G-1=\\ \dim_{\D_{\p H_1}}\p I_{H_1}+\dim_{\D_{\p H_2}} \p I_{H_2}-1.\end{multline*}

 On the other side, the quotient $M/L$ is isomorphic to $\p I_{H_1}^G/\p I_A^G$, which is $\D_{\p G}$-torsion-free by Proposition \ref{Atorfree}.
Again, by Lemma  \ref{quotientdim},   \[\dim_{\D_{\p G}} M/L=\dim_{\D_{\p G}} \p I_{H_1}^G-\dim_{\D_{\p G}}\p  I_A^G=\dim_{\D_{\p H_1}} \p I_{H_1}-1.\] 

Therefore, the short exact sequence $0\rar L\rar M\rar M/L\rar 0$ of $\p G$-modules satisfies the requirements of Lemma \ref{sestorfree}; and the conclusion follows. 
\end{proof}

We now turn to the study of torsion-free modules that have the form of an augmentation ideal of a cyclic HNN extension. 

Our point is to prove that a certain module, which may have the form 
\[R^m /R(u_1, \dots, u_m)\]
for some group ring $R$,
is  $\D_R$-torsion-free. When extending this $R$-module with coefficients in a bigger ring, some $u_i$ may become invertible. The following elementary lemma simply studies this scenario, which we shall encounter many times.

\begin{lem} \label{freequo} Let $R$ be a unital ring and  $M$ be an $R$-module. Let $m_0\in M$ and let $u$ be a unit of $R$. Then there is an isomorphism of $R$-modules
\[\ga: \frac{M\oplus R}{(m_0, u)}\lrar M \]
given by 
\[\ga (m, r)=m-ru^{-1}m_0,\]
with inverse 
\[\ga^{-1}(m)=(m, 0).\]
\end{lem}

We can now state  the second result of this section. 
\begin{pro}\label{hnntor} Let $H\leq G$ be subgroups of $\FF$. Suppose that we can write $G=N\rtimes (t)=\lan N, t\ran$ for some $H\leq N\leq G$ and some $t\in G$. Let $u\in H$ be an element which generates a maximal abelian subgroup $\lan u\ran $ in $H$ and suppose that $v=tut^{-1}\in H$. Then the $\p G$-module $M$ defined by 
\[M=\frac{\p I_H^G\oplus \p G}{\p G\,  (v-1-t(u-1), v-1)}\]
is $\D_{\p G}$-torsion-free. 
\end{pro}

\begin{proof} The proof is relatively long and, for convenience, it is divided in several intermediate claims. 

Let  $R$ be the subring of $\D_{\p G}$ generated by $\D_{\p N}$ and $t$. The structure of this ring is explained in Subsection \ref{universalsf}.

\begin{claim} \label{lemskewdia0} The natural map $$ R  \tens{\p G}M \to  \D_{\p G} \tens{\p G}M$$ is injective. 
\end{claim}
 
\begin{proof}
Since $v\in H$, $v-1$ is invertible in $\D_{\p[N]}$ (and hence in $R$). Directly from the definition of $M$, we can rewrite, using Lemma \ref{freequo},  
\begin{equation} \label{eq:lemskew} R\tens{\p G} M\cong \frac{\left(R\otimes_{\p[H]} I_H\right)\oplus R}{R\,  \left(v-1-t(u-1), v-1\right)}\cong R\tens{\p G} \p I_H^G \end{equation}

Similarly, $\D_{\p G} \tens{\p G}M\cong   \D_{\p G} \tens{\p G} \p I_H^G$. Observe that 
$$
R\tens{\p G}\p I_H^G\cong  R\tens{\p H} \p I_H \textrm{\ and\ }  \D_{\p G} \tens{\p G} \p I_H^G \cong  \D_{\p G} \tens{\p H} \p I_H.
$$

Since $N\unlhd G$ is normal, $R$ is a right $\D_{\p[N]}$-module. In particular, it is a right $\D_{\p[H]}$-module. So $R$ is a direct  $\D_{\p H}$-summand of  $\D_{\p G} $ and it follows directly that the map 

$$ R\tens{\p H} \p I_H\to  \D_{\p G} \tens{\p H} \p I_H$$ is injective. This proves the claim.
\end{proof}

\begin{claim}
We have the following equality of subsets of $\D_{\p N} \tens{\p N} \p I_H^N$,
\begin{equation} \label{int1}  1\tens{\p N} \p I_H^N\, \, \bigcap\, \,  \D_{\p N}\tens{\p N} (u-1)=1\tens{\p N} \p N (u-1).  \end{equation}
\end{claim} 
 
 \begin{proof}
 By assumption, $\lan u\ran$ is a maximal abelian subgroup of $H$. By Proposition \ref{Atorfree}, the $\p N$-module 
\[M_0=\frac{\p I_H^N}{\p N (u-1)}\]
is $\D_{\p N}$-torsion-free. This implies that the canonical map 
\[M_0\lrar \D_{\F_p N} \tens{\p N} M_0\]
is injective. This gives the claim.\end{proof}

We denote by $H^{t^n}$ the conjugation $t^n Ht^{-n}$ (which may not be the standard way). The $H^{t^n}$ are also subgroups of $N$ having $\lan u^{t^n}\ran $ as maximal abelian subgroup. 

\begin{claim}  For all $n\in \Z$, we have the following equality of subsets of $\D_{\p N}\tens{\p N}  \p I_G$,
\begin{equation} \label{int3} 1\tens{\p N} \p t^n I_H^N\, \, \bigcap \, \, \D_{\p N}\tens{\p N} t^n(u-1)=1\tens{\p N}\p t^n N(u-1).\end{equation}
\end{claim}
\begin{proof}
The same way we had the equality (\ref{int1}), we can derive, for the same reasons, 
\begin{equation} \label{int11}  1\tens{\p N} \p I_{H^{t^{-n}}}^N\, \, \bigcap\, \,  \D_{\p N}\tens{\p N} (u^{t^{-n}}-1)=1\tens{\p N} \p N (u^{t^{-n}}-1).  \end{equation}
Notice that $\p I_G$ has a right $\lan t\ran$-module structure  by multiplication. This induces a right $\lan t\ran$-module structure on $\D_{\p N}\tens{\p N}  \p I_G$. Since $N\n G$, then $t$ normalises $N$. We have the equations of subsets of $\D_{\p N}\tens{\p N}  \p I_G$:
\[\left(1\tens{\p N} \p N (u-1)\right)\, t^m= 1\otimes \p \, t^m N (u^{t^{-m}}-1),\]
and 
\[\left(1\tens{\p N} \p I_H^N \right)\, t^{m} = 1\otimes \p \, t^m  I_{H^{t^{-m}}}.\]
As a consequence, applying the multiplication-by-$t^n$ automorphism of $\D_{\p N}\tens{\p N}  \p I_G$ to the equation  (\ref{int11}),  we get (\ref{int3}). \end{proof}

Notice the following decomposition of $\p N$-modules
\[\p I_H^G=\bigoplus_{n\in \Z}\, \, \p \, t^n I_H^N,\]
which yields to the following decomposition of $\p$-vector spaces 
\[\D_{\p N}\tens{\p N} \p I_H^G\cong \bigoplus_{n\in \Z} \, \D_{\p N} \tens{\p N}\, \p\, t^n  I_H^N.\]

\begin{claim} 
We have the following equation of subsets of $\D_{\p N}\tens{\p N} \p I_H^G$,
\begin{equation} \label{int2} 1\tens{\p N} \p I_H^G\, \, \bigcap \, \, \D_{\p N} \tens{\p N} \p G(v-1-t(u-1))=1\tens{\p N} \p G(v-1-t(u-1)). \end{equation}
\end{claim}
\begin{proof}
Let us take an element $w$ that belongs to the left-hand side. This element will have the form 
\[w=\sum_{k=n_1}^{n_2}c_k\otimes t^k(v-1-t(u-1)),\, \, \mbox{for some $c_k\in \D_{\p N}$,}\]
and will also belong to $1\tens{\p N} \p I_H^G$.  We rewrite 
\[w= c_{n_1}\otimes t^{n_1}(v-1)+\sum_{k=n_1+1}^{n_2} \left(c_k\otimes t^k(v-1) -c_{k-1}\otimes t^{k}(u-1)\right) +c_{n_2}\otimes t^{n_2+1}(u-1).\]
Since $w\in 1\tens{\p N} \p I_H^G $, we can look at the highest power $t^{n_2+1}$ to deduce that 
\[c_{n_2}\otimes t^{n_2+1}(u-1)\in  1\tens{\p N} \p t^{n_2+1} I_H^N\, \, \bigcap \, \, \D_{\p N}\tens{\p N} t^{n_2+1}(u-1).\]
By (\ref{int3}), this implies that 
\[c_{n_2}\otimes t^{n_2+1}(u-1)\in 1\tens{\p N}\p t^{n_2+1} N(u-1),\]
so $c_{n_2}\in \p N$. 
Let $n_1+1\leq k\leq n_2$. Inspecting again the expression of $w$ at the component with power $t^k$, we have that \[c_k\otimes t^k(v-1) -c_{k-1}\otimes t^{k}(u-1)\in  1\tens{\p N} \p t^k I_H^N.\]
If we knew that  $c_k\in \p N$, then it would follow that 
\[ c_{k-1}\otimes t^{k}(u-1)\in 1\tens{\p N} \p t^k I_H^N\, \, \bigcap \D_{\p N}\tens{\p N} t^{k}(u-1).\]
By (\ref{int3}), this means that 
\[c_{k-1}\otimes t^{k}(u-1) \in 1\tens{\p N}\p t^n N(u-1),\]
and this implies that $c_{k-1}\in \F_p N$. 

We have proven that if $c_k\in \F_p N$, for $n_1<k\leq n_2$, then $c_{k-1}\in \p N$. Since we also know that $c_{n_2}\in \p N$, an inductive argument gives that  $c_k\in \p N$ for every $k$, meaning that  \[w\in 1\tens{\p N} \p N[t^{\pm}, \sigma]\,  (v-1-t(u-1))= 1\tens{\p N} \p G(v-1-t(u-1)).\]
This proves that 
\[ 1\tens{\p N} \p I_H^G\, \, \bigcap \, \, \D_{\p N} \tens{\p N} \p G(v-1-t(u-1))\sub 1\tens{\p N} \p G(v-1-t(u-1)),\]
one inclusion of (\ref{int2}). The reverse inclusion is trivial, so equation (\ref{int2}) is proven. \end{proof}
\begin{claim} \label{al}
The natural map $$\al:M\to R\tens{\p G}  M$$
 is injective. 
\end{claim} 
\begin{proof}
As discussed during Subsection \ref{universalsf}, there is a canonical isomorphism of $\p N$-modules
\[ R\cong \D_{\p N}[t^{\pm 1}, \sigma] \cong \D_{\p N} \tens{\p N} \p N[t^{\pm}, \sigma]=  \D_{\p N} \tens{\p N} \p G.  \]
This extends to a canonical isomorphism of $\p N$-modules
\[\psi: \frac{R \tens{\p G} \p I_H^G }{R\tens{\p G} \p G(v-1-t(u-1))}\lrar \frac{\D_{\p N} \tens{\p N} \p I_H^G}{ \D_{\p N}\tens{\p N} \p G(v-1-t(u-1))}. \]
Therefore, there is a commutative triangle of canonical $\p N$-homomorphisms 
\begin{equation*}
\begin{tikzcd}
    \frac{\p I_H^G}{\p G (v-1-t(u-1))} \ar[d, "\al'"] \ar[dr, "\eta"] & \\
     \frac{R \tens{\p G} \p I_H^G }{R\tens{\p G} \p G(v-1-t(u-1))} \ar[r, "\psi"] & \frac{\D_{\p N} \tens{\p N} \p I_H^G}{ \D_{\p N}\tens{\p N} \p G(v-1-t(u-1))},\\
\end{tikzcd}
\end{equation*}
The canonical map $\eta$ is injective due to (\ref{int2}). Since $\psi$ is an isomorphism, this implies that $\al'$ is injective. 

Let $(x, y)\in \p I_H^G\oplus \p G$ be such that $(x,y)+ {\p G\,  (v-1-t(u-1), v-1)}$ belongs to the kernel of $\al$. Then $1\otimes (x, y)=c \otimes (v-1-t(u-1), v-1)$ for some $c\in R$. Notice that 
\[1\otimes x=c\otimes (v-1-t(u-1))\in 1\tens{\p G} \p I_H^G\, \bigcap\,  R\tens{\p G} (v-1-t(u-1)).\]
From the injectivity of $\al'$, this implies that 
\[c\otimes (v-1-t(u-1))\in 1 \tens{\p G} \p G(v-1-t(u-1)),  \]
so $c\in \p G$ and then $(x, y)\in \p G(v-1-t(u-1), v-1).$
Thus $\al$ is injective and  the claim is demonstrated.\end{proof}

From Claims \ref{al} and \ref{lemskewdia0}, we conclude that $M$ is $\D_{\p G}$-torsion-free. \end{proof}

\section{Proofs of the main results}
\label{finalsect}

\subsection{Amalgamated products}\label{paraamal}


Now we are ready to prove Theorem \ref{amalgamparafreeintr}.

\begin{teo} \label{amalgamparafree} Let $\ti{H_1}$ and $\ti{H_2}$ be finitely generated groups. Let $1\neq u_1\in \ti{H_1}$ and let $1\neq u_2\in \ti{H_2}$. Consider the following amalgamated product of cyclic amalgam
\[\ti{G}=\ti{H_1}\underset{u_1=u_2}{*}\ti{H_2}\cong \frac{\ti{H_1}* \ti{H_2}}{\lan \lan u_1u_2^{-1}\ran \ran}.\]
Then $\ti{G}$ is parafree if and only if the three  following conditions hold.
\begin{enumerate}
\item $\ti{H_1}$ and $\ti{H_2}$ are parafree.
    \item The element $u_1u_2^{-1}$ of $\ti{H_1}* \ti{H_2}$ is not a proper power in the abelianization of $\ti{H_1}* \ti{H_2}$.
    \item There is at least one $i\in \{1, 2\}$ such that $u_i$ is not a proper power in $\ti{H_i}$.
\end{enumerate} 

\end{teo}
\begin{rem*} The condition (2) can be substituted by the condition
\begin{enumerate}
\item [(2') ]$\abr(\ti{G})=\abr(\ti{H_1})+\abr(\ti{H_2})-1$. \end{enumerate}
The proof of the theorem shows that the condition (3) can be substituted by the condition
\begin{enumerate}
\item [(3') ] All the centralizers of non-trivial elements in $\ti G$ are cyclic.\end{enumerate}

\end{rem*}
\begin{proof} We first prove  that the conditions are necessary. Let us assume that $\ti{G}$ is parafree.

Both $\ti{H_1}$ and $\ti{H_2}$ are subgroups of $\ti{G}$ and hence they are residually nilpotent.
We want  to show that, for all primes $p$, $\ti{H_1}_{\hat p}$ and $\ti{H_2}_{\hat p}$ are free. By Proposition \ref{carparafree}, this would imply that  $\ti{H_1}$ and $\ti{H_2}$ are parafree.

Fix a prime $p$. Recall that the rank of a pro-$p$ group $\GG$ equals the rank of its $p$-abelianisation $\GG/[\GG, \GG]\GG^p$. From this, we observe that 
\begin{equation} \label{eq1} d(\ti{G}_{\hat p})\ge d(\ti{H_1}_{\hat p}) +d(\ti{H_2}_{\hat p})-1.\end{equation}
Consider the closure $\overline{H_i}$ of the image of  $\ti{H_i}$  under the canonical map $\ti{G}\to  \ti{G}_{\hat p}$. Since $\tilde G$ is parafree, $\ti{G}_{\hat p}$ is free pro-$p$, and so both $\overline{H_1}$ and $\overline{H_2}$ are free pro-$p$, because any closed subgroup of a free pro-$p$ group is again free. Note that  $\ti{H_1}$ and $\ti{H_2}$ generate $\ti{G}$. Hence the canonical map  $f:\overline{H_1}\coprod \overline{H_2}\to \ti{G}_{\hat p}$ is onto. 
 
Since $\ti{G}$ is parafree,  the images of $u_1$ and $u_2$ in $\ti{G_{\hat p}}$ are non-trivial. Therefore, $\ker f\ne \{1\}$. 
We also recall that the groups $\overline{H_1}\coprod \overline{H_2}$ and $\ti{G}_{\hat p}$ are free pro-$p$. So, by the Hopfian property, for the surjection $f$ between free pro-$p$ groups to have non-trivial kernel we must have 
\begin{equation} \label{eq2} d(\ti{G}_{\hat p})\le  d\left(\overline{H_1}\coprod \overline{H_2}\right)-1= d(\overline{H_1}) +d(\overline{H_2})-1. \end{equation}

Moreover, since $\overline{H_i}$ is a quotient of $\ti{H_i}_{\hat p}$, $d(\overline{H_i})\leq d(\ti{H_i}_{\hat p})$. The latter observation in combination with (\ref{eq1}) and (\ref{eq2}) yields $d(\overline{H_i})= d(\ti{H_i}_{\hat p})$. The pro-$p$ group $\overline{H_i}$ is free, so the canonical map $\ti{H_i}_{\hat p}\to \overline{H_i}$ is an isomorphism. Hence $\ti{H_i}_{\hat p}$ is free pro-$p$ and the first condition is proved.

The previous argument also shows that, for all primes $p$,
$$d(\ti{G}_{\hat p})= d(\ti{H_1}_{\hat p}) +d(\ti{H_2}_{\hat p})-1, $$ and this implies   the second condition.

Finally, suppose that $u_1=v_1^{n_1}$ in $\ti{H_1}$, and that $u_2=v_2^{n_2}$ in $\ti{H_2}$  with $n_1,n_2\geq 2$.  
 Since $u_i\neq 1$ and $\ti{H_i}$ is torsion-free, then   $\langle v_i\rangle$ are infinite cyclic. Consider the subgroup $A$ generated by $v_1$ and $v_2$. 
 By Proposition \ref{subgraph}, $A$  is isomorphic to $\langle v_1\rangle *_{v_1^{n_1}=v_2^{n_2}} \langle v_2\rangle$.
  This group  is non-abelian and belongs to the centralizer of $u_1=u_2$ in $\ti {G}$. This contradicts Proposition \ref{centralizer}. This shows that the third condition holds.

We now verify that the three given conditions are sufficient. There is a canonical isomorphism 
\begin{equation} \label{eq4} \ti{G}_{\ab}\cong \frac{\ti{H_1}_{\ab}\oplus \ti{H_2}_{\ab}}{u_1-u_2}.\end{equation}
From the fact that $u_1-u_2$ is not a proper power in $\ti{G}_{\ab}$, we see that $\ti{G}_{\ab}$ is torsion-free of rank $\abr(\ti{H_1})+\abr(\ti{H_2})-1$. Therefore for any prime $p$, \[d(\ti{G}_{\hp})=d(\ti{H_1}_{\hp})+d(\ti{H_2}_{\hp})-1.\]

Let us fix an arbitrary prime $p$ from this point on.  
Consider the canonical map $\phi: \ti{G}\lrar \ti{G}_{\hat{p}}$. We denote $G=\phi(\ti{G})$. 

\begin{claim} \label{amal1} The pro-$p$ group $\ti{G}_{\hat{p}}$ is free.
\end{claim}

\begin{proof} Since $\ti{H_1}$ and $\ti{H_2}$ are parafree, the pro-$p$ groups $\ti{H_1}_{\hp}$ and $\ti{H_2}_{\hp}$ are free. Thus  $\ti{G}_{\hp}$ is a quotient of the free pro-$p$ group
$\ti{H_1}_{\hp}\coprod \ti{H_2}_{\hp}$ by the closed normal subgroup generated by $u_1u_2^{-1}$. Recall that an element $u$ of a pro-$p$ group $\GG$ is primitive if and only if it is primitive in the $p$-abelianisation $\GG/[\GG, \GG]\GG^p$. By assumption, the element $u_1u_2^{-1}\in H_1*H_2$ is primitive in the abelianisation, so it must be  primitive in its pro-$p$ completion $\ti{H_1}_{\hp}\coprod \ti{H_2}_{\hp}$. Hence, $\ti{G}_{\hp}$  is free pro-$p$.
\end{proof}
 
\begin{claim} \label{amal2} The restrictions of $\phi$ to each $\ti{H_i}$ are injective. 
\end{claim}
\begin{proof}
To check this claim, we consider each restriction \[\phi_i=\phi|_{\ti{H_i}}: \ti{H_i}\lrar H_i,\] where $H_i=\phi(\ti{H_i})$. The  subgroups  $\overline{H_1}$ and $\overline{H_2}$ of the free pro-$p$ group $\ti{G}_{\hat{p}}$  are closed. Hence they both are free pro-$p$ groups. 

Since the induced ${\phi_{i}}_{\hat{p}}: \ti{H_i}_{\hat{p}}\lrar \overline{H_i}$ are surjective maps of free pro-$p$ groups, 
\begin{equation} \label{amal2eq1} d(\ti{H_i}_{\hat{p}})\geq d(\overline{H_i}), \, \, \, \mbox{for all $i\in \{1, 2\}.$}\end{equation}
Furthermore, by the universal property of the coproduct, there is a continuous homomorphism
\[f: \overline{H_1}\coprod \overline{H_2}\lrar \ti{G}_{\hat{p}}\]
which sends $\overline{H_i}$ to each corresponding copy in $\ti{G}_{\hat{p}}$. Notice that $\im f$ contains both $H_1$ and $H_2$, so $G\leq \im f$. Since $G$ is dense in $\ti{G}_{\hat{p}}$, $f$ is surjective. By assumption, the element $u_1u_2^{-1}\in H_1*H_2$ is primitive in its abelianisation. So, for some $i$, $u_i$ is non-trivial in the abelianisation of $G$, and hence $u_i$ is non-trivial in $\bar{H_i}$. In particular, $u_1u_2^{-1}$ is a non-trivial element in the kernel of $f$. Arguing exactly as we did in deriving equations (\ref{eq1}) and (\ref{eq2}), we conclude that the induced maps ${\phi_{i}}_{\hat{p}}: (\ti{H_i})_{\hp}\lrar {H_i}_{\hp}$   are isomorphisms. This implies that the surjections $\phi_i$   are injective because the groups $\ti{H_i}$ are residually-$p$.
\end{proof}
\begin{claim} \label{amal3} The map $\phi: \ti{G}\lrar \ti{G}_{\hat{p}}$ is injective. 
\end{claim}
\begin{proof}
We want to apply Proposition  \ref{pro-$p$ embedding lemma} to the map $\phi: \ti{G}\lrar G\sub \ti{G}_{\hp}$. 

Notice first that, by Claim \ref{amal2},
the kernel of $\phi$ intersects trivially the subgroups $\ti{H_1}$ and $\ti{H_2}$. Therefore,  by the Bass-Serre theory, the kernel acts freely on a tree, and so it is free. 

We consider the corresponding 
$\p G$-module 
\[M=\p G\tens{\p \ti{G}} \p I_{\ti{G}}.\]

Consider $A=\lan \phi(u_1)\ran$ and  \[J=\{(x, -x): x\in \p I_A^{G}\}\leq \p I_{H_1}^G\oplus \p I_{H_2}^G.\]
Then, by Proposition \ref{straugm}(1),

\begin{equation} \label{eq3} M\cong \frac{\p I_{H_1}^G\oplus \p I_{H_2}^G}{J}.\end{equation}
Note that we did not use the fact that $G$ is an amalgamated product of $H_1$ and $H_2$ along the cyclic subgroup $\lan \phi(u_1)\ran$ (which, in fact, we do not know yet to be true). We derived the isomorphism of (\ref{eq3}) applying Proposition \ref{straugm}(1) to the splitting of the group $W=\ti{G}$ to obtain the corresponding expression for its augmentation ideal $\F_p I_{\ti{G}}$. The expression for $M$ is now derived from its very definition $M=\p G\tens{\p \ti{G}} \p I_{\ti{G}}.$

Without loss of generality, we suppose that $u_1$ is not a proper power in $\ti{H_1}$. 
By Claim \ref{amal2}, $ \ti{H_i}\cong  H_i$, so $H_i$ ($i=1,2$) are parafree and  $\phi(u_1)$ is not a proper power in ${H_1}$. Hence $A=\lan u_1\ran $ is a maximal abelian subgroup of $H_1$.

By Proposition \ref{amaltor}, $M$ is $\D_{\p G}$-torsion-free with dimension 
\[ \dim_{\D_{\p G} }M =  \dim_{\D_{\p H_1}} \p I_{H_1}+\dim_{\D_{\p H_2}}\p I_{H_2}-1.\]
In addition, by Proposition \ref{Bettibound}, 
 \[  \dim_{\D_{\p H_1}} \p I_{H_1}+\dim_{\D_{\p H_2}}\p I_{H_2}-1=d({H_1}_{\hat{p}})+d({H_2}_{\hat{p}})-1 = d(\ti{G}_{\hat{p}}).\]
It follows that $\dim_{\D_{\p G}} M=d(\ti{G}_{\hp})$. We can apply Proposition \ref{pro-$p$ embedding lemma} to conclude that $\phi: \ti{G}\lrar G$ is injective. \end{proof}

The last claim implies that $\ti{G}$ is residually nilpotent. Moreover, we already know that each $\ti{G}_{\hat{p}}$ is free, so  $\ti{G}$ is parafree by Proposition \ref{carparafree}. 
\end{proof}

\subsection{HNN extensions}\label{parahnn}

Now we   prove Theorem \ref{hnnparafreeintr}.

\begin{teo} \label{hnnparafree} Let $\ti{H}$ be a finitely generated  group. Let $u, v\in \ti{H}\setminus \{1\}$. Consider the following cyclic HNN extension of $\ti{H}$ 
\[\ti{G}= \frac{\ti{H}* \lan t\ran}{\lan \lan tut^{-1}v^{-1}\ran \ran}. \]
Then $\ti{G}$ is  parafree if and only if the four following conditions hold.
\begin{enumerate}
\item The group $\ti{H}$  is parafree.
    \item The element $uv^{-1}$ is not a proper power in $\ti{H}_{\ab}$.
    \item At least one of $u$ or $v$ is not a proper power in $\ti{H}$.
    \item  The image of the element $u$ is non-trivial in some finite nilpotent quotient of $\ti{G}$.
\end{enumerate}

\end{teo}
\begin{rem*} The condition (2) can be substituted by the condition
\begin{enumerate}
\item [(2') ]$\abr(\ti{G})=\abr(\ti{H})$. \end{enumerate}
The proof of the theorem shows that the condition (3) can be substituted by the condition
\begin{enumerate}
\item [(3') ] All the centralizers of non-trivial elements in $\ti G$ are cyclic.\end{enumerate}

\end{rem*}

\begin{proof} First let us show that the given conditions are necessary. Assume that $\ti{G}$ is parafree.

The group $\ti{H}$ is a subgroup  of $\ti{G}$ and hence it is residually nilpotent.
We want  to show that, for all primes $p$, $\ti{H}_{\hat p}$ is  free. By Proposition \ref{carparafree}, this would imply that  $\ti{H}$ is  parafree.

Fix a prime $p$. Observe that 
$$d(\ti{G}_{\hat p})\ge d(\ti{H}_{\hat p}).$$
Consider the closure $\overline{H}$ of the image of  $\ti{H}$  under the canonical map $\ti{G}\to  \ti{G}_{\hat p}$. Since $\tilde G$ is parafree, $\ti{G}_{\hat p}$ is free pro-$p$, and so $\overline{H}$  is also  free pro-$p$. Note that  $\ti{H}$ and $t$ generate $\ti{G}$. Hence,  the canonical map  $f:\overline{H}\coprod \Z_p\to \ti{G}_{\hat p}$, which sends the generator 1 of  $\Z_p$ to $t$, is onto. 

Since $\ti{G}$ is parafree,  the images of $u$ and $v$ in $\ti{G}_{\hat p}$ are non-trivial. Therefore, $\ker f\ne \{1\}$.
Thus, since the groups $\overline{H}\coprod \Z_p$ and $\ti{G}_{\hat p}$ are free pro-$p$,
$$d(\ti{G}_{\hat p})\le d(\overline{H}).$$

Since $\overline{H}$ is a quotient of $\ti{H}_{\hat p}$, $d(\overline{H})=d(\ti{H}_{\hat p})$, and since, $\overline{H}$ is free pro-$p$, the canonical map $\ti{H}_{\hat p}\to \overline{H}$ is an isomorphism. Hence, $\ti{H}_{\hat p}$ is free pro-$p$, and the first condition is proved.

The previous argument also shows that for all prime $p$,
$$d(\ti{G}_{\hat p})= d(\ti{H}_{\hat p}).$$
This implies   the second condition.

 Now suppose that $u=w^{n_1}$  and  $v=w_2^{n_2}$ in $\ti{H}$ with $n_1,n_2\geq 2$.   Since $u, v\neq 1$ and $\ti{H}$ is torsion-free, then   $\langle w_i\rangle$ are infinite cyclic. 
 Consider the subgroup $A$ of $\ti{G}$ generated by $w_1$,  $w_2$ and $t$.
 By Proposition \ref{subgraph}, $A$  is isomorphic to the HNN extension $A^\prime=\langle w_1, w_2, t| tw_1^{n_1}t^{-1}=w_2^{n_2}\rangle$.  The centralizer of $w_1^{n_1}=(t^{-1}w_2t)^{n_2}$ in $A^\prime$ contains $w_1$ and $t^{-1}w_2t$, so it is not abelian.  Thus the centralizer of $u$ in $\ti {G}$ is not abelian.
  This contradicts Proposition \ref{centralizer}. This shows that the third condition holds.
  The forth condition holds because $\ti{G}$ is parafree.

Now we are going to verify that these four conditions are sufficient for $\ti{G}$ to be parafree. 
First of all, it is clear that 
\[\ti{G}_{\ab}\cong \frac{\ti{H}_{\ab}}{\lan u-v\ran}\oplus \lan t\ran.\]
From the fact that $u-v$ is not a proper power in $\ti{H}_{\ab}$, we see that $\ti{G}_{\ab}$ is torsion-free of the same rank as $\ti{H}_{\ab}$. 
Therefore,   $d(\ti{G}_{\hp})=d(\ti{H}_{\hp})$.
In addition, we also see that $t$ is primitive in $\ti{G}_{\ab}$. Let us fix a  prime $p$ such that  the image of the element $u$ is non-trivial in some finite $p$-quotient of $\ti{G}$.
Consider the canonical map $\phi: \ti{G}\lrar \ti{G}_{\hat{p}}$.

\begin{claim} \label{hnn0} The pro-$p$ group $\ti{G}_{\hp}$ is free and the element $\phi(t)$ is primitive in $\ti{G}_{\hp}$.
\end{claim}
\begin{proof} The pro-$p$ group $\ti{G}_{\hp}$  is the quotient of $ \ti{H}_{\hat{p}}\coprod \Z_p $ by the closed subgroup generated by $tut^{-1}v^{-1}$. Since $ \ti{H}$ is parafree,   $ \ti{H}_{\hat{p}}$, is free pro-$p$. Thus $ \ti{H}_{\hat{p}}\coprod \Z_p $ is free pro-$p$. Observe also that the element $tut^{-1}v^{-1}$ is primitive in $ \ti{H}_{\hat{p}}\coprod \Z_p $ (since it is primitive in its $p$-abelianisation). Hence, $\ti{G}_{\hp}$  is free pro-$p$.
 \end{proof}

We name $H=\phi(\ti{H})$. 

\begin{claim}\label{hnn1} The restriction of $\phi$ to $\ti{H}$ is injective. 
\end{claim}
\begin{proof}
To verify this, consider the closed subgroup $\overline{H}\leq \ti{G}_{\hat{p}}$. Since $\ti{G}_{\hat{p}}$ is free, the pro-$p$ group $\overline{H}$ must be free. We notice that the epimorphism $\phi: \ti{H}\lrar H$ induces a continuous epimorphism $\phi_{\hat{p}}: \ti{H}_{\hat{p}}\lrar \overline{H}$. 
In particular,  
\begin{equation} \label{n>d} d(\ti{H}_{\hp})\geq d(\overline{H}).\end{equation} 
Furthermore, by the universal property of the coproduct, there is a continuous homomorphism
\[f: \overline{H}\coprod \Z_p \lrar \ti{G}_{\hat{p}},\]
which sends $\overline{H}$ to $\ti{G}_{\hat{p}}$, by inclusion; and   $\Z_p$ to the cyclic pro-$p$ group generated by $\phi(t)$. Since the image of $f$ contains both $H$ and $\phi(t)$, it follows that $\im f$ contains $\phi(\ti{G})$, so $f$ must be a surjective. In addition, it has a nontrivial kernel; since $\phi(t)\phi(u)\phi(t)^{-1}\phi(v)^{-1}=1$ and $\phi(u)\neq 1$, by assumption. Here we have used   that the canonical map $\iota: \overline{H}* \Z_p\lrar  \overline{H}\coprod \Z_p$ is injective.

This verifies that $f$ is a surjective and non-injective continuous homomorphism  between free pro-$p$ groups. Hence,   
\[d(\overline{H})+1=d(\overline{H})+d(\Z_p)=d\left( \overline{H}\coprod \Z_p \right)> d(\ti{G}_{\hat{p}}). \]
This, in addition to (\ref{n>d}), implies that $d(\ti{H}_{\hp})=d(\overline{H})$. So  $\phi_{\hat{p}}: \ti{H}_{\hat{p}}\rar \overline{H}$ is a continuous epimorphism between free pro-$p$ groups of the same rank, which implies that it is  an isomorphism.  Since $\ti{H}$ is parafree, it is residually-$p$. Hence $\phi$ is also injective.
  \end{proof}
 
\begin{claim}\label{hnn3} The map $\phi: \ti{G}\lrar \ti{G}_{\hp}$ is injective.
\end{claim}
\begin{proof}
We denote $G=\phi(\ti{G})$. We want to apply Proposition \ref{pro-$p$ embedding lemma} to the map $\phi: \ti{G}\lrar G\sub \ti{G}_{\hp}$. We already know, from Claim \ref{hnn0}, that  $\ti{G}_{\hp}$ is free. Notice also  that, by Claim \ref {hnn1},
the kernel of $\phi$ intersects trivially the subgroup $\ti{H}$ of $\ti{G}$. A standard  Bass-Serre theoretic argument proves that the kernel is free.

We define a continuous homomorphism  $q: \ti{G}_{\hat{p}}\lrar \Z_p$ such that $q(\phi(t))=1$, and $q(\phi(h))=0$ if $h\in \ti{H}$. The restriction $q|_G$ verifies that its its kernel $\ker q|_G$ contains $H=\phi(\ti{H})$. 

We now consider the $\p G$-module 
\[M=\p G\tens{\p \ti{G}} \p I_{\ti{G}}.\]
We apply Proposition \ref{straugm}(2) in a similar way to how we obtained  (\ref{eq3}) to deduce the following isomorphism of $\p[G]$-modules:
\[M\cong \frac{\p I_{H}^G\oplus \p G}{\p G\, (\phi(v)-1-\phi(t)(\phi(u)-1),\,  \phi(v)-1)}.\]

Without loss of generality, we suppose that $u$ is not a proper power in $\ti{H}$. 
Since $\phi: \ti{H}\lrar H$ is an isomorphism, $H$ is parafree and $\phi (u)$ is not a proper power in $H$. From Proposition \ref{centralizer}, we deduce that that  $\lan \phi(u)\ran$ is a maximal abelian subgroup of $H$. By Proposition \ref{hnntor}, this implies that $M$ is $\D_{\p G}$-torsion-free.  

Using Lemma \ref{freequo}, similarly as in  (\ref{eq:lemskew}), we have the following isomorphisms of left $\D_{\p G}$-modules
\[\D_{\p G}\tens{\p G} M\cong \D_{\p G}\tens{\p G} \p I_H^G\cong \D_{\p G}\tens{\p H} \p I_H\cong \D_{\p G}\tens{\D_{\p H}} \left( \D_{\p H}\tens{\p H} \p I_H \right).\]
The combination of these isomorphisms with Proposition \ref{Bettibound} yields to
\[\dim_{\D_{\p G}} M= \dim_{\D_{\p H} }\p I_H=d(\ti{H}_{\hat{p}})=d(\ti{G}_{\hat{p}}).\]
We can apply Proposition \ref{pro-$p$ embedding lemma} to conclude that $\phi: \ti{G}\lrar G$ is injective.
\end{proof}

The last claim implies that $\ti{G}$ is residually nilpotent. Moreover, we already know that every pro-$p$ completion $\ti{G}_{\hat{p}}$ is free, so $\ti{G}$ is parafree by Proposition \ref{carparafree}. 
\end{proof}

We are now ready to prove Corollary \ref{hnncor}, where we the fourth condition of Theorem \ref{hnnparafree} is replaced by a very simple one.

\begin{proof}[Proof of Corollary \ref{hnncor}]
It is not hard to see that the four given conditions  imply the conditions of \ref{hnnparafreeintr}. In fact, under these assumptions,  the image of $u$ is non-trivial in the abelianisation of $W$.

In order to prove that they are necessary,  let us suppose that  $W$ is parafree. By Theorem \ref{hnnparafree}, all the properties $(1)-(3)$ hold. For the sake of a contradiction, we assume that $(4)$ does not hold. Then, by $(2)$, the image of $uv^{-1}$ in $U_{ \ab}$ generates a subgroup which contains both $u$ and $v$. So there exists $c\in \Z$ such that $u\equiv (uv^{-1})^c \mod [U, U]$. In addition, since $U_{\hp}$ is free of rank 2 and the image of $uv^{-1}$ is primitive in $U_{\hp}$; we can denote by $\NN$ the closed normal subgroup generated by $uv^{-1}$ in $U_{\hp}$ and observe that $U_{\hp}/\NN\cong \Z_p$. Thus $[U, U]\sub \NN$ and, in particular, $u\in  \NN$. 

Denote by $\NN_1$ the closed normal subgroup of $W_{\hp}$ generated by $uv^{-1}$. Then $u\in \NN_1$ by the previous argument. Now observe that $uv^{-1}=[u^{-1}, t^{-1}]$. We are going to verify that $u\in \bigcap_k \ga_k W_{\hp}=\{1\}.$ This would result in the contradiction $u=1$, finishing the proof. We proceed inductively to check the claim. The base of the induction is trivial because $u\in \ga_1 W_{\hp}=W_{\hp}$. If $u\in \ga_k  W_{\hp}$, then $uv^{-1}=[u^{-1}, t^{-1}]\in [\ga_k W_{\hp}, W_{\hp}]=\ga_{k+1} W_{\hp}$. Since $\ga_k W_{\hp}$ is a closed normal subgroup of $W_{\hp}$, then $\NN_1\sub \ga_{k+1} W_{\hp}$ and hence $u\in \ga_{k+1} W_{\hp}$, completing the inductive step. In conclusion, we have proven that $1\neq u\in W$ is trivial in $W_{\hp}$, which gives the desired contradiction.
\end{proof}

\subsection{The fundamental group of graph of groups} In this subsection we prove  Corollary \ref{fundamentalparafreeintr} from the introduction. Recall that this gives a general criterion for a graph of groups with cyclic edge groups to be parafree.

\begin{proof}[Proof of Corollary \ref{fundamentalparafreeintr}]

First, we assume that $W$ is parafree. We want to show that the four conditions of the corollary hold. The conditions (3) and (4) are a consequence of $W$ being parafree. Now we  show the conditions (1) and (2) arguing by induction on the number of edges $|E(\Gamma)|$.  If $|E(\Gamma)|=0$, the claim is obvious.

 Recall that every spanning tree of a simplicial graph $\Ga$ has $|V(\Ga)|-1$ edges, and hence $H_1(\Ga; \Q)$ has rank $1-\chi(\Ga).$ This explains the contribution of the underlying graph to the abelian rank in equation (2) of Corollary \ref{fundamentalparafreeintr}.

Now we assume that we have proved that  (1) and (2) hold if $|E(\Gamma)|\le n $ and we consider the case $|E(\Gamma)|=n+1$. 
Let $e\in E(\Gamma)$.

If $\Gamma \setminus \{e\}=\Delta_1\cup \Delta_2$ is disconnected then, by Proposition \ref{subgraph},
 $$W\cong U*_{\mathcal G(e)} V, \textrm{\ where\ }U\cong \pi_1(\mathcal G, \Delta_1) \textrm{\ and\ } V\cong \pi_1(\mathcal G, \Delta_2).$$
If $\mathcal G(e)=\{1\}$, then $W\cong U*V$, and so, since $W$ is parafree,  $U$ and $V$ are also parafree. If $\mathcal G(e)\cong \Z$, then $U$ and $V$ are parafree by Theorem \ref{amalgamparafree}. Since $|E(\Delta_1)|,|E(\Delta_2)|\le n $, we can apply the induction. Thus, the four conditions of Corollary \ref{fundamentalparafreeintr} hold for $U$ and $V$. This implies immediately that conditions (1) and (2) also hold  for $W$.

If $\Gamma \setminus \{e\}=\Delta$ is  connected, then, by Proposition \ref{subgraph},
 $$W\cong U*_{t_e} \textrm{\ where\ }U\cong \pi_1(\mathcal G, \Delta).$$
If $\mathcal G(e)=\{1\}$, then  $W\cong U*\Z$, and so,  since $W$ is parafree, $U$ is also  parafree. If $\mathcal G(e)\cong \Z$, then $U$ is  parafree by Theorem \ref{hnnparafree}.
Since $|E(\Delta)|\le n $, we can apply the induction. Thus, the four conditions of Corollary \ref{fundamentalparafreeintr} holds for $U$. This implies immediately that the conditions (1) and (2) holds also  for $W$.

Now we assume that the four conditions of the corollary hold. We want to show that $W$ is parafree. We will argue by induction on the number of edges $|E(\Gamma)|$.  If $|E(\Gamma)|=0$, the claim is obvious.

For the induction hypothesis, we suppose that we have shown that  $W$ is parafree  if $|E(\Gamma)|\le n $ and  we consider the case $|E(\Gamma)|=n+1$. Let $e\in E(\Gamma)$. 

If $\Gamma \setminus \{e\}=\Delta_1\cup \Delta_2$ is disconnected and so $\chi(\Gamma)=\chi(\Delta_1)+\chi(\Delta_2)-1.$ Furthermore, by Proposition \ref{subgraph},
 $$W\cong U*_{\mathcal G(e)} V, \textrm{\ where\ }U\cong \pi_1(\mathcal G, \Delta_1) \textrm{\ and\ } V\cong \pi_1(\mathcal G, \Delta_2).$$

 We have that
\begin{multline*}
\abr(W)\ge \abr(U)+\abr(V)-\abr(\mathcal G(e))\ge\\
 \sum_{v\in V(\Delta_1)}\abr(\mathcal G(v))-\sum_{f\in E(\Delta_1)} \abr(\mathcal G(f))+1 -\chi(\Delta_1)+\\
\sum_{v\in V(\Delta_2)}\abr(\mathcal G(v))-\sum_{f\in E(\Delta_2)} \abr(\mathcal G(f)) +1-\chi(\Delta_2) -\abr(\mathcal G(e))= \\
 \sum_{v\in V(\Gamma)}\abr(\mathcal G(v))-\sum_{f\in E(\Gamma)} \abr(\mathcal G(f)) +1-\chi(\Gamma)
 \end{multline*}
 
Since $\abr(W)= \sum_{v\in V(\Gamma)}\abr(\mathcal G(v))-\sum_{f\in E(\Gamma)} \abr(\mathcal G(f))+1 -\chi(\Gamma)$, we obtain that
\begin{multline*}
\abr(U)=\sum_{v\in V(\Delta_1)}\abr(\mathcal G(v))-\sum_{f\in E(\Delta_1)} \abr(\mathcal G(f)) +1-\chi(\Delta_1),\\
\abr(V)=\sum_{v\in V(\Delta_2)}\abr(\mathcal G(v))-\sum_{f\in E(\Delta_2)} \abr(\mathcal G(f))+1 -\chi(\Delta_2)  \textrm{\ and\ }\\
\abr(W)=\abr(U)+\abr(V)-\abr(\mathcal G(e)).
\end{multline*}
Thus,  $U$ and $V$ satisfy the four conditions of the corollary and, by the  inductive hypothesis, they are parafree. If $\mathcal G(e)=\{1\}$, then $W\cong U*V$, and so,  $W$ is parafree. If $\mathcal G(e)\cong \Z$, then $W$ is parafree by Theorem \ref{amalgamparafree} and the remark afterwards.

 If $\Gamma \setminus \{e\}=\Delta$ is  connected then $\chi(\Gamma)=\chi(\Delta)+1.$ by Proposition \ref{subgraph},
 $$W\cong U*_{t_e} \textrm{\ where\ }U\cong \pi_1(\mathcal G, \Delta).$$
 
 We have that
\begin{multline*}
\abr(W)\ge \abr(U)-\abr(\mathcal G(e))+1\ge\\
 \sum_{v\in V(\Delta)}\abr(\mathcal G(v))-\sum_{f\in E(\Delta)} \abr(\mathcal G(f))+1 -\chi(\Delta) -\abr(\mathcal G(e))+1= \\
 \sum_{v\in V(\Gamma)}\abr(\mathcal G(v))-\sum_{f\in E(\Gamma)} \abr(\mathcal G(f))+1 -\chi(\Gamma).
 \end{multline*}
 
By assumption, $\abr(W)= \sum_{v\in V(\Gamma)}\abr(\mathcal G(v))-\sum_{f\in E(\Gamma)} \abr(\mathcal G(f))+1 -\chi(\Gamma)$. So we deduce that
\begin{multline*}
\abr(U)=\sum_{v\in V(\Delta)}\abr(\mathcal G(v))-\sum_{f\in E(\Delta)} \abr(\mathcal G(f)) -\chi(\Delta) \textrm{\ and\ }\\
\abr(W)=\abr(U)-\abr(\mathcal G(e))+1.\end{multline*}
Thus, $U$ satisfies the four conditions of the corollary and, by the  inductive hypothesis, it is parafree. If $\mathcal G(e)=\{1\}$, then $W\cong U*\Z$, and so,  $W$ is parafree. If $\mathcal G(e)\cong \Z$, then $W$ is parafree by Theorem \ref{hnnparafree} and the remark afterwards.

\end{proof}

\end{document}